\documentclass[11pt,oneside,reqno]{amsart}

\usepackage{amssymb,amscd,stmaryrd}
\usepackage[mathcal]{eucal}
\usepackage{array,float}
\usepackage{enumitem}
\usepackage{mathtools}
\usepackage{xcolor}

\usepackage{xy}
\input xy
\xyoption{all}
\usepackage{pdflscape}
\usepackage{hyperref}
\usepackage[left=3cm, right=3cm]{geometry}

\numberwithin{equation}{section}
\numberwithin{equation}{subsection}

\newtheorem{thm}{Theorem}[section]

\newtheorem{lemma}[thm]{Lemma}

\newtheorem{definition}[thm]{Definition}
\newtheorem*{remark*}{Remark}

\newtheorem{example}[thm]{Example}

\newcommand{\ind}{{\text{Ind}}}
\newcommand{\Hom}{{\mathrm{Hom}}}

\newcommand{\tra}{{\mathrm{tra}}}
\newcommand{\res}{{\mathrm{res}}}

\newcommand{\Ho}{{\mathrm{H}}}

\newcommand{\Ind}{\mathrm{Ind}}

\newcommand{\bigzero}{\mbox{\normalfont\Large\bfseries 0}}
\newcommand{\rvline}{\hspace*{-\arraycolsep}\vline\hspace*{-\arraycolsep}}

\newcommand{\bZ}{{\mathbb Z}}

\newcommand{\irr}{\mathrm{Irr}}

\title[On Projective representations of direct product of  groups]{On Projective representations of direct product of  groups}

\author{Sumana Hatui}
\address{School of Mathematical Sciences,
	National Institute of Science Education and Research, An OCC of Homi Bhabha National Institute, Bhubaneswar 752050, Odisha, India}

\email{sumanahatui@niser.ac.in, sumana.iitg@gmail.com}

\begin{document}

\subjclass[2020]{20C25, 20D15, 20F18, 20K25}
\keywords{Schur multiplier, Projective representations, Representation group, $p$-groups, Direct product of groups}
\begin{abstract}
		Let $G=G_1 \times  G_2$ be a finite group.   
	By  \cite[Theorem 2.3, p. 20]{GK2},  the second cohomology group
	$\Ho^2(G,\mathbb C^\times)$ is isomorphic to $\Ho^2(G_1,\mathbb C^\times) \times \Ho^2(G_2,\mathbb C^\times) \times \Hom(G_1/G_1' \otimes_\mathbb Z G_2/G_2', \mathbb C^\times ).$
A $2$-cocycle $\alpha$ of $G$ is called a bilinear cocycle  if  the corresponding cohomology class $[\alpha]$ of $ \Ho^2(G,\mathbb C^\times)$
lies in  $\Hom(G_1/G_1' \otimes_\mathbb Z G_2/G_2', \mathbb C^\times)$. 
In this article, our aim is to construct an irreducible complex  projective representation $\rho$ of $G$ for bilinear cocycles $\alpha$.

If $G_1$ is any abelian $p$-group  and $G_2$ is an elementary abelian $p$-group, then we  give a construction of  $\rho$  for bilinear cocycles $\alpha$ of $G$. For a subgroup $H$ of $G$ of index $\leq p^2$, we  also count the number of cohomology classes $[\alpha]$  for which the irreducible projective representations behave the same  while restricting on $H$. 
Finally,  we consider any $p$-group $G=G_1\times G_2$,
	 and we discuss how the above construction helps us to describe an irreducible $\alpha$-representation of $G$ when $[\alpha]$ is of order $p$ or $G_2/G_2'$ is elementary abelian.
	 We also discuss several examples as an application of the above results.
\end{abstract}

\maketitle 

\section{Introduction}
For a finite group $G$, the second cohomology group $\Ho^2(G,\mathbb C^\times)$ is called the Schur multiplier of $G$.
For a $2$-cocycle  $\alpha$  of $G$, $[\alpha]\in \Ho^2(G,\mathbb C^\times)$ denotes the cohomology class of $\alpha$ contains all  the $2$-cocycles cohomologous  to $\alpha$.  
A complex projective representation of $G$ corresponding to the $2$-cocycle $\alpha$  is a map $\rho: G \to GL(V)$ satisfying
\[
\rho(1)=id_V, ~\rho(gh)=\alpha(g,h)\rho(g)\rho(h), g,h \in G,
\]
where $V$ is a complex vector space. Then we say $\rho$ is an $\alpha$-representation of $G$; see \cite[Chapter 3]{GK} for the details.
 $\irr^\alpha(G)$ denotes the set of all  irreducible $\alpha$-representations of $G$ up to linear equivalence.

If $G=G_1\times G_2$, then 
by  \cite[Theorem 2.3, p. 20]{GK2}, we have
\begin{eqnarray}\label{Schur}
\Ho^2(G,\mathbb C^\times)\cong \Ho^2(G_1,\mathbb C^\times) \times \Ho^2(G_2,\mathbb C^\times) \times \Hom(G_1/G_1' \otimes_\mathbb Z G_2/G_2', \mathbb C^\times ).
\end{eqnarray}
A $2$-cocycle $\alpha$ of $G$ is called bilinear cocycle if 
$[\alpha]$ of  $\Ho^2(G,\mathbb C^\times)$ belongs to $  \Hom(G_1/G_1' \otimes_\mathbb Z G_2/G_2', \mathbb C^\times )$ via this isomorphism.
For convenience,  we use the notation  $ G_1\otimes G_2,$ instead of the abelian tensor product $G_1/G_1' \otimes_\mathbb Z G_2/G_2'$ without further reference.
If the cohomology class $[\alpha] \in  \Ho^2(G,\mathbb C^\times)$ belongs to $ \Ho^2(G_1,\mathbb C^\times) \times \Ho^2(G_2,\mathbb C^\times)$ via the above isomorphism, then the elements of $\irr^\alpha(G)$ are linearly equivalent to the elements of the form $(\rho_1\otimes \rho_2)$ for $\rho_i\in \irr^{\alpha_i}(G_i)$, for $[\alpha_i]=\res^G_{G_i}([\alpha])$, $i=1,2$; see \cite[Proposition 7.1,  p. 124]{GK2}. 
But this fact is not known if $[\alpha]$ belongs to $\Hom(G_1/G_1' \otimes G_2/G_2', \mathbb C^\times)$.  Hence, in this article, we give a construction of	the elements of $\irr^\alpha(G)$ for the $p$-groups $G=G_1 \times G_2$ for bilinear cocycles $\alpha$ of $G$. 
If $G$ is a nilpotent group, say,  $G=S_1 \times S_2\times \cdots S_k$, for Sylow subgroups $S_i$ of $G$, then
any projective representation of $G$ is projectively equivalent to  $(\rho_1\otimes \rho_2\otimes \cdots \otimes\rho_k)$ for $\rho_i$ an irreducible projective representation  of $S_i$; see \cite[Corollary 2.2.11]{karpilovsky} and \cite[Proposition 7.4, p. 125]{GK2}. Hence it is enough to study the projective representations of $p$-groups.

The theory of projective representations  was first started by pioneering work of Schur in \cite{IS4, IS7, IS11}. Later several authors have studied this problem for different groups in \cite{Moa1, Moa, Moa2, Na1,Na2, WB, PS,  hatui2023projective, hatui2023projective1}.
Let us recall the history of projective representations of finite abelian groups.
We know that the irreducible complex ordinary representations of finite abelian groups are one dimensional, whereas this fact is not true for their projective representations. 
This problem has been studied by several authors, most notably by Morris, Saeed-ul-Islam, and Thomas in \cite{Moa, Moa2}. 
 For some special class of cocycles $\alpha$, all irreducible $\alpha$-representations of $(\mathbb Z/n\mathbb Z)^k$  and of  finite abelian groups had been studied  in \cite{Moa} and \cite{Moa2}  respectively.
 Later, Higgs \cite{RH} gives a method of construction of the elements of $\irr^\alpha(G)$ for the elementary abelian $p$-groups $G$, for cocycle $\alpha$ of $G$.
 
Continuing this line of investigation, for the bilinear cocycles $\alpha$, 
 we give an explicit construction of an irreducible $\alpha$-representation of abelian $p$-groups $G=G_1\times G_2$ when $G_2$ is elementary abelian; see section \ref{abelian}.   We also describe  alternate methods in some special cases. 
It follows  by \cite[Theorem  2.21, p. 380]{GK3} that it is enough to construct just one element of $\irr^\alpha(G)$.
  
 Let $H$ be a subgroup of $G$ of index $p$. We know by Clifford's theory that for $\rho \in \irr^\alpha(G)$,
 either $\rho |_H$ is irreducible or it is the direct sum of $p$ many distinct elements of $\mathrm{Irr}^{\alpha |_{H\times H}} (H)$.
In section \ref{restriction}, we count the number of cohomology classes $[\alpha]$ for bilinear cocycles $\alpha$, for which the irreducible $\alpha$-representations  behave in the same way while restricting on $H$. We also extend these results when $H$ is a subgroup of $G$ of index $p^2$.

In section \ref{directproduct}, we consider  any $p$-group $G=G_1\times G_2$. We discuss how the construction  given in section \ref{abelian} for the abelian groups helps us to  describe an irreducible $\alpha$-representation of $G$ for bilinear cocycles $\alpha$  when $[\alpha]$ is of order $p$ or $G_2/G_2'$ is elementary abelian.

We also discuss some examples as an application of these constructions.
%
%
%
%

\section{Preliminaries}
In this section, we recall the results which will be used in the upcoming sections.
Let 
\[
1\to N \to G  \xrightarrow{\pi} G/N \to 1 
\]
be a central extension such that $N \subseteq G'$. 
Then,
by Hochschild Serre spectral sequence \cite[Theorem 2, p. 129]{HS}, we have the following exact sequence
\begin{eqnarray}\label{Hochschild}
1\to \Hom(N, \mathbb C^\times)  \xrightarrow{\tra} \Ho^2(G/N, \mathbb C^\times) \xrightarrow{\inf} \Ho^2(G, \mathbb C^\times),
\end{eqnarray}
where $\tra: \Hom(N, \mathbb C^\times)  \xrightarrow{\tra} \Ho^2(G/N, \mathbb C^\times)$  is called the transgression homomorphism, which is defined by $\tra(\chi)=[\alpha]$ for $\chi \in \Hom(N, \mathbb C^\times)$ such that
\[
\alpha(xN,yN)=\chi\big(\mu(xN)\mu(yN)\mu(xyN)^{-1}\big), x,y \in G,
\]
$\text{for a section } \mu \text{ of } \pi$. The map 
$ \inf: \Ho^2(G/N, \mathbb C^\times) \to \Ho^2(G, \mathbb C^\times)$  is called the inflation homomorphism, which is defined as follows:
$
\inf([\beta])=[\alpha]
$ for  $ [\beta] \in \Ho^2(G/N, \mathbb C^\times)$ such that $\alpha(x,y)=\beta(xN,yN), x,y \in G$.

\bigskip

Let $G=G_1\times G_2$.
By  \cite[Theorem 2.3, p. 20]{GK2}, there is an isomorphism 
$$\psi:  \Ho^2(G_1,\mathbb C^\times) \times \Ho^2(G_2,\mathbb C^\times) \times \Hom(G_1 \otimes G_2, \mathbb C^\times ) \to \Ho^2(G,\mathbb C^\times),$$ where $\psi$ is defined as follows: for $[\alpha_i]\in \Ho^2(G_i,\mathbb C^\times), i=1,2$ and $f\in \Hom(G_1\otimes G_2,\mathbb C^\times)$,
$\psi([\alpha_1], [\alpha_2], f)=[\alpha]$ such that 
\[
\alpha(g_1g_2, g_{11}g_{22})=\alpha_1(g_1, g_{11})\alpha_2(g_2, g_{22})f(g_{11} G_1' \otimes g_2G_2'), ~g_1,g_{11} \in G_1, g_2,g_{22} \in G_2.
\]
Hence, a bilinear $2$-cocycle $\alpha$ of $G$, we mean $[\alpha] \in \Ho^2(G,\mathbb C^\times)$ such that  $[\alpha]=\psi(1,1,f) $ for some $f\in  \Hom(G_1\otimes G_2, \mathbb C^\times )$, i.e., $\alpha$ is of the form 
\begin{eqnarray}\label{bilinearcocycle}
\alpha(g_1g_2, g_{11}g_{22})=f(g_{11} G_1' \otimes g_2G_2'), ~g_1,g_{11} \in G_1, g_2,g_{22} \in G_2.
\end{eqnarray}

\begin{definition}[$\alpha$-regular element of $G$]
For a $2$-cocycle $\alpha$ of $G$, an element $g \in G$ is called $\alpha$-regular element of $G$ if $\alpha(g,x)=\alpha(x,g)$ for all such $x \in G$ satisfying $xg=gx$.
\end{definition}
\begin{definition}[$\alpha$-central type]
For a $2$-cocycle $\alpha$ of $G$, $G$ is called $\alpha$-central type if $G$ has a unique irreducible $\alpha$-representation upto linear equivalence.\\
\end{definition}

\subsection{Mackey's Irreducibility criterion}
Let  $\alpha$ be a $2$-cocycle of  $G$ and $H$ be a subgroup of $G$. Suppose $\rho: H \to GL(V)$ is an $\alpha|_{H \times H}$-representation of $H$, where $\alpha|_{H\times H}$ denotes the restriction of  $\alpha$ on  $H$.
  For $g \in G$,  the map $\rho^{(g)}: H \to GL(V)$ for $g \in G$ defined by 
\[
\rho^{(g)}(x)=\alpha^{-1}(xg, g^{-1})\alpha^{-1}(g, g^{-1}xg)\alpha(g,g^{-1})\rho(g^{-1}xg),
\]
is an $\alpha$-representation of $G$,  follows from \cite[Lemma 7.1, p. 226]{GK}. 
Let $\{g_1,g_2, \ldots, g_n\}$ be a transversal of $H$ in $G$. Then the induced representation 
$\Gamma=\Ind_H^G(\rho)$ is defined as a  matrix representation  $\Gamma(g)= [\Gamma_{ij}(g)], ~ 1\leq i, j \leq n$, where
$$
\Gamma_{ij}(g)=\alpha(g,g_j)\alpha^{-1}(g_i,g_i^{-1}gg_j)\rho(g_i^{-1}gg_j) \text{ if } g_i^{-1}gg_j\in H, 
$$
and 
$ \Gamma_{ij}(g)=0$ if $g_i^{-1}gg_j\notin H$; see  \cite[Theorem 2.2, p. 201]{GK} for the details.\\

Now we recall Mackey's criterion \cite[Corollary 7.12, p. 239]{GK} for the irreducibility of the induced representation.

\begin{thm}\label{Mackey}
Let  $\alpha$ be a $2$-cocycle of $G$, $H$ be a normal subgroup of $G$ and $T$ be a transversal of $H$ in $G$. Let  $\rho: H \to GL(V)$ be an irreducible $\alpha|_{H \times H}$-representation of $H$. Then $\ind^G_H(\rho)$ is irreducible if and only if $\rho$ and $\rho^{(t)}$ are disjoint for all $t \in T-\{1\}$.
\end{thm}

%
%
%
%
%
%

\section{Construction of $\alpha$-representations for abelian $p$-groups} \label{abelian}
Let  \begin{eqnarray}\label{elementaryabelian}
	G=G_1 \times  G_2
	\end{eqnarray}
such that 
\[
G_1=\langle x_1 \rangle \times \langle x_2 \rangle \times \ldots  \times \langle x_t \rangle \cong \mathbb Z/p^{r_1}\mathbb Z \times \mathbb Z/p^{r_2}\mathbb Z \times \cdots \times  \mathbb Z/p^{r_t}\mathbb Z, t \geq 1, r_1 \geq r_2 \ldots \geq r_t \geq 1\]
and
\[
G_2=\langle y_1 \rangle \times \langle y_2 \rangle \times \ldots  \times \langle y_s \rangle \cong (\bZ/p\bZ)^{s}, s \geq 1.
\]
Suppose  $\alpha$ is a bilinear cocycle of $G$. Our aim is to construct an irreducible $\alpha$-representation of $G$. It follows  by \cite[Theorem  2.21, p. 380]{GK3} that it is enough to construct  one irreducible $\alpha$-representation of $G$.
In this section, $G$ is always of the above form, and $\alpha$ is a bilinear cocycle of $G$. We use these notations without further reference.

By \eqref{bilinearcocycle}, it follows that every bilinear cocycle $\alpha$ of $G$ is cohomologous to a cocycle of the following form.
For $0 \leq m_i, m_i' <  p^{r_i}$ and $0 \leq k_i, k_i'< p$, 
\begin{eqnarray*}
 && \alpha\big(x_1^{m_1}x_2^{m_2}\ldots x_t^{m_t}y_1^{k_1}y_2^{k_2}\ldots y_s^{k_s},
x_1^{m'_1}x_2^{m'_2}\ldots x_t^{m'_t}y_1^{k'_1}y_2^{k'_2}\ldots y_s^{k'_s}\big)\\
=&& \prod_{1 \leq i \leq t, 1\leq   j \leq s}\mu_{ij}^{m_i'k_j} \text{ such that  } \mu_{ij} \in \mathbb C^\times, \mu_{ij}^p=1.
\end{eqnarray*}
Fix $\xi=e^{\frac{2\pi i}{p}}$.  
Then by the above  description we have
\begin{eqnarray}\label{bilinea}
	&& \alpha\big(x_1^{m_1}x_2^{m_2}\ldots x_t^{m_t}y_1^{k_1}y_2^{k_2}\ldots y_s^{k_s},
	x_1^{m'_1}x_2^{m'_2}\ldots x_t^{m'_t}y_1^{k'_1}y_2^{k'_2}\ldots y_s^{k'_s}\big)\nonumber
	\\ 
	=&& \prod_{1 \leq i \leq t, 1\leq   j \leq s} \xi^{c(i,j)m_i'k_j} \text{ for  } 0 \leq c(i,j) \leq p-1.
\end{eqnarray}
The bilinear cocycles $\alpha$ determine $p^{st}$ many cohomology classes $[\alpha] \in \Ho^2(G, \mathbb C^\times)$ and those are explicitly given by \eqref{bilinea} for $p^{st}$ many choices of $c(i,j)$.

\bigskip

Let $M_{m\times n}(\mathbb Z/p\mathbb Z)$ denotes  $m\times n$ matrices with entries in $\mathbb Z/p\mathbb Z$.
Now the following two facts follow from the discussion given  in \cite[Lemma 2.23, p. 382]{GK3} and \cite[p. 772]{RH}.\\

\textbf{Fact 1}: Every  $\alpha$ of the  form \eqref{bilinea} uniquely determines  and determined by a matrix $A_\alpha \in  M_{(s+t) \times (s+t)}(\mathbb Z/p\bZ)$ of the form
%
%
\[
A_\alpha
=
\begin{pmatrix}
 \bigzero 
 & \rvline & 
\begin{matrix}
c(1,1) & c(1,2) & \cdots  & c(1,s)\\
c(2,1) & c(2,2) & \cdots  & c(2,s)\\
\vdots & \vdots & \vdots & \vdots \\
c(t,1) & c(t,2) & \cdots  & c(t,s)
\end{matrix}\\
\hline
\begin{matrix}
-c(1,1) & -c(2,1) & \cdots  & -c(t,1)\\
-c(1,2) & -c(2,2) & \cdots  & -c(t,2)\\
\vdots & \vdots & \vdots & \vdots \\
-c(1,s) & -c(2,s) & \cdots  & -c(t,s)\\
\end{matrix}
& \rvline  &  \bigzero &
\end{pmatrix}_{(s+t) \times (s+t)}.
\]
We write as  \[
A_\alpha
=
\left(
\begin{array}{c|c}
0 & A_{1\alpha}\\
\hline
-A_{1\alpha}^t & 0
\end{array}
\right).
\]
\\
 
\textbf{Fact 2}: 
An element $g=x_1^{m_1}x_2^{m_2}\ldots x_t^{m_t}y_1^{k_1}y_2^{k_2}\ldots y_s^{k_s},~ 0 \leq m_i, k_j \leq p-1$ of $G$ is an $\alpha$-regular element  if and only if $A_\alpha b=0 \pmod p $ for $b^T=[m_1,m_2, \ldots, m_t, k_1, k_2, \ldots, k_s]$. Observe that any element of the subgroup $\langle x_1^p \rangle \times \langle x_2^p \rangle \times\ldots \langle x_t^p\rangle$ is  an $\alpha$-regular element of $G$.
 
We denote the set of all $\alpha$-regular elements of $G$ by $G_0$.

\begin{thm} With the notations given above, the following holds.
\begin{enumerate}[label=(\roman*)]	
	\item $1 \leq |G/G_0| \leq p^{2\min\{s,t\}}$.
		\item For any $\rho \in \irr^\alpha(G)$, $1\leq dim(\rho) \leq p^{\min\{s,t\}}$.
				\item For $\rho \in \irr^\alpha(G)$,  dim $\rho=p^k$ if and only if rank $A_{1\alpha}=k$. 
						\item For any fix $\alpha$, all elements of $\irr^\alpha(G)$ have the same dimension.
		\item   $G$ is $\alpha$-central type if and only if $s=t$, $r_i=1, 1\leq i \leq t$ and $A_{1\alpha}$ has  rank $t$.
\end{enumerate}
\end{thm}
	\begin{proof}
	Let $X=\prod_{i=1}^t x_i^{m_i} \in G$, where $m_i$ are multiple of $p$ for  $i=1,2, \ldots, t$. 
	Then  $\alpha(X,Y)=\alpha(Y,X)=1$ for all $Y \in G$. Hence the subgroup $S=\langle x_1^p \rangle \times \langle x_2^p \rangle \times \ldots  \times \langle x_t^p \rangle $  of $G$ is contained in $G_0$.
An element $g=x_1^{m_1}x_2^{m_2}\ldots x_t^{m_t}y_1^{k_1}y_2^{k_2}\ldots y_s^{k_s} \in G$, $1\leq m_i, k_i\leq p-1$ is  an $\alpha$-regular element of $G$ if and only if $A_\alpha b=0 \pmod p $ for $b^t=[m_1,m_2, \ldots, m_t, k_1, k_2, \ldots, k_s]$ (by Fact 2).	
	Hence
	\[
	 \text{rank} ~A_\alpha=(s+t)- \text{dim}~(\text{ker}~ A_\alpha)=2k, \text{ where }  \text{rank} ~A_{1\alpha}=k.
	 \]
Thus $|G/G_0|=p^{(s+t)- \text{dim}(\text{ker}~ A_\alpha)}=p^{2k}.$
It follows from \cite[Theorem 2.21, p. 380]{GK3} that, 
for each $\alpha$,  there are $|G_0|$ many irreducible $\alpha$-representations upto linear equivalence,  all of dimension $\sqrt{|G/G_0|}=p^k$. 
Hence $(iii)$ and  $(iv)$ follows.

Using the fact $0\leq k \leq \min\{s,t\}$, $(i)$ and  $(ii)$ follows.

$G$ is $\alpha$-central type  if and only if $G_0=1.$  Then we must have $S=1$, i.e., $r_i=1$ for all $i$. 
Now $\text{ker}~ A_\alpha=1$ if and only if $A_{1\alpha}$ is of full rank, i.e., $s=t$ and $A_{1\alpha}$ must be of rank $t$. Hence $(v)$  follows.
		\end{proof}

%
In the upcoming subsections, we describe different methods for constructing irreducible complex $\alpha$-representations of $G$.		
		
		\subsection{Rank $A_{1\alpha}=n, n \leq \min\{s,t\}$}\label{construction}
 Consider a bilinear cocycle  $\alpha$ of $G$ such that $A_{1\alpha}$ has rank $n$.  Let $[a_{i,j}]$ be the row reduced echelon form of $A_{1\alpha}$ such that the pivot elements  $1$ occur in the columns $k_1, k_2, \cdots, k_n$, $k_1<k_2\cdots <k_n$. 
Now using \textbf{Fact 2} we observe that the following subgroup of $G$ 
\begin{eqnarray*}
	&& 
	\prod_{j=1, j \neq k_1, k_2, \ldots, k_n}^s \Big\langle \prod_{l=1}^{n} y_{k_l}^{-a_{l,j}}y_j \Big\rangle 
\end{eqnarray*}
is contained in  $G_0$.
Since rank $A_\alpha=2n$, degree of irreducible $\alpha$-representations are $p^n$.
Consider the  subgroup $G_3=\langle  x_1, x_2, \ldots, x_t\rangle \times \prod_{j=1, j \neq k_1, k_2, \ldots, k_n}^s \langle \prod_{l=1}^{n} y_{k_l}^{-a_{l,j}}y_j \rangle$ of $G$ of  index $p^n$.

Let $X, X' \in \langle  x_1, x_2, \ldots, x_t\rangle $ and $Y, Y' \in \prod_{j=1, j \neq k_1, k_2, \ldots, k_n}^s \langle \prod_{l=1}^{n} y_{k_l}^{-a_{l,j}}y_j \rangle$. Since $Y,Y'$ are $\alpha$-regular elements of $G$,  we have 
\begin{eqnarray*}
\alpha(XY,X'Y')&=&\alpha(Y,X')=\alpha(X',Y)=1.
\end{eqnarray*}
Therefore
$	[\alpha]|_{G_3\times G_3}=1$.
Consider the trivial representation of $G_3$, say $\rho$.  The set $T=\{\prod_{i=1}^n y_{k_i}^{r_i}, 0\leq r_i \leq p-1\}$ is a transversal of $G_3$ in $G$. Now we prove the following lemma, and then using Theorem \ref{Mackey}, we have $\ind_{G_3}^G (\rho)$ is an irreducible $\alpha$-representation of $G$ of dimension $p^n$.

\begin{lemma} 
	$\rho^{(z)} \neq \rho$ for any $z \in T- \{1\}$.
\end{lemma}
\begin{proof}
Let  $z=\prod_{i=1}^n y_{k_i}^{r_i}$ for $0 \leq r_i \leq p-1$. Since $z\neq 0$, there is a $q$ such that   $r_q \neq 0$. For $1 \leq i \leq n$, each column $k_i$  of the matrix $A_{1\alpha}$ contains an element, say $c(j_i,k_i)$, which is non-zero. So $\alpha(y_{k_i},x_{j_i})=\xi^{c(j_i,k_i)}\neq 1$.
We have $\rho^{(z)}(g)=\alpha^{-1}(z,g)\rho(g), g \in G$. 

To prove our result, it is enough to find an element $g$  such that $\alpha(z,g)\neq 1$.  If $g=\prod_{j=1}^t x_{j}^{m_j}$, then
\[
\alpha(z,g)=\alpha(\prod_{i=1}^n y_{k_i}^{r_i}, \prod_{j=1}^t x_{j}^{m_j})=\xi^{\sum_{j=1}^{t}c(j,k_1)r_1m_j+\sum_{j=1}^{t}c(j,k_2)r_2m_j+\sum_{j=1}^{t}c(j,k_n)r_nm_j}.
\]
Consider the $ n \times t$ matrix $C=[c_{ij}]$ such that $c_{ij}=c(j, k_i)r_i$. Then $C\neq 0$. Consider the system of equations $Cx=b$, where $x=[m_1, m_2, \ldots m_t]^T$ and $b=[0,  \ldots, 0, r_q,0 \ldots, 0]^T$. 
It is enough to show the existence of a solution of the system of equations $Cx=b$. If it exists, say  $x=[m_1', m_2', \ldots m'_t]^T$, then for $g=\prod_{j=1}^t x_{j}^{m'_j}$, we have  $\alpha(z,g)=\xi^{\text{sum of the elements of } b}=\xi^{r_q} \neq 1$.
\\
\textbf{claim:}  $Cx=b$ has a solution.\\
Observe that,  for $1\leq i \leq n$, the  $i$-th row of $C$ is the $r_i$ times $k_i$-th column of $A_{1\alpha}$. 
Let $B$ be the row reduced echelon form of $A_{1\alpha}$.
Since the $k_i$-th column of $B$ is the vector $\underbrace{[0,  \ldots, 0, 1,0 \ldots, 0]^T}_\textrm{1 in the i-th position}$,  the column space of $C$ will be spanned by the set $\{[0,  \ldots, 0, r_i,0 \ldots, 0]^T, 1 \leq i \leq n\}$. Hence the result follows.
\end{proof}

Now we discuss the following examples as an application of the above construction.
\begin{example}\label{example1}
Let $G=\mathbb Z/p^{r_1}\mathbb Z \times \mathbb Z/p^{r_2}\mathbb Z \times \mathbb Z/p\mathbb Z \times \mathbb Z/p\mathbb Z =\langle x_1\rangle \times \langle x_2\rangle \times \langle y_1\rangle \times  \langle y_2 \rangle $. 

(1) Consider a bilinear cocycle of the form
\[
\alpha(x_1^{m_1}x_2^{m_2}y_1^{k_1}y_2^{k_2}, x_1^{m_1'}x_2^{m_2'}y_1^{k_1'}y_2^{k_2'} )=\xi^{m_1'(k_1+k_2)}, \xi=e^{\frac{2\pi i}{p}}.
\]
Then $A_{1\alpha}=\begin{pmatrix}1&1\\0&0\end{pmatrix}$, and it is easy to check that
\[
G_0=\langle x_1^p, x_2, y_1y_2^{-1} \rangle.
\]
By section \ref{construction}, consider $G_3=\langle x_1, x_2, y_1y_2^{-1} \rangle$, and $T=\{y_1^{r}, 0 \leq r \leq p-1\}$ is a transversal of $G_3$ in $G$. Then $\tau=\mathrm{Ind}_{G_3}^G(\rho)$ is an irreducible $\alpha$-representation of $G$ of dimension $p$ which is defined on the generators as a $p\times p$ matrix  as follows.\\

$\tau(x_1)$ is a diagonal matrix with $\tau(x_1)_{mm}=\xi^{-(m-1)}, 1\leq m \leq p$, $\tau(x_2)=I_p$, and
 \begin{equation*}
	 \tau(y_i)_{mn}  = \begin{cases}
	1 &\text{if  $m-n=1 \pmod p,$}\\
	0 &\text{otherwise.}
	\end{cases}
	\end{equation*}
 
 \bigskip
 
 (2) Consider another bilinear cocycle of $G$
\[
\alpha(x_1^{m_1}x_2^{m_2}y_1^{k_1}y_2^{k_2}, x_1^{m_1'}x_2^{m_2'}y_1^{k_1'}y_2^{k_2'} )=\xi^{(m_1'+m_2')k_1}, \xi=e^{\frac{2\pi i}{p}}.
\]
Then Then $A_{1\alpha}=\begin{pmatrix}1&0\\1&0\end{pmatrix}$ and 
\[
G_0=\langle x_1^p, x_2^p, x_1x_2^{-1}, y_2 \rangle.
\]
By section \ref{construction}, consider $G_3=\langle x_1, x_2, y_2 \rangle$, and $T=\{y_1^{r}, 0 \leq r \leq p-1\}$ is a transversal of $G_3$ in $G$. Then $\tau=\mathrm{Ind}_{G_3}^G(\rho)$ is an irreducible $\alpha$-representation of $G$ of dimension $p$ which is defined on the generators as a $p\times p$ matrix  as follows.\\

$\tau(x_i)$ is a diagonal matrix with $\tau(x_i)_{mm}=\xi^{-(m-1)}, 1\leq m \leq p$,	$ \tau(y_2)=I_p$ and
 \begin{equation*}
	 \tau(y_1)_{mn}  = \begin{cases}
	1 &\text{if  $m-n=1 \pmod p,$}\\
	0 &\text{otherwise.}
	\end{cases}
	\end{equation*}

\end{example}

\bigskip

\subsection{$s=t$, each row and each column of $A_{1\alpha}$ contains exactly one  non-zero element. }\label{s=t}
Using the above method we can construct an irreducible $\alpha$-representation in this case. 
Now we discuss an alternate method here, which enable us to write the $\alpha$-representation explicitly; see Theorem \ref{alternate}. 

Consider the bilinear cocycles $\alpha$ such that $c(i,j_i)\neq 0, 1\leq i \leq t$ and $j_1, j_2, \cdots,j_t$ are all distinct.
Write $G$ as 
\begin{eqnarray*}
	G&=&
	\big(\mathbb Z/p^{r_1}\mathbb Z \times \mathbb Z/p\mathbb Z \big)  \times 
	\big(\mathbb Z/p^{r_2}\mathbb Z \times  \mathbb Z/p\mathbb Z \big)  \times 
	\cdots \times 
	\big(\mathbb Z/p^{r_t}\mathbb Z \times \mathbb Z/p\mathbb Z  \big) \\
	&=&\big(\langle x_1 \rangle  \times  \langle y_{j_1} \rangle \big)
	\times \big(\langle x_2 \rangle \times \langle y_{j_2}  \rangle \big) \times  \cdots \times 
	\big(\langle x_t \rangle  \times  \langle y_{j_t} \rangle\big ).
\end{eqnarray*}
Then each such $\alpha$ is cohomologous to the cocycles of the form
\begin{eqnarray*}\label{bilinear1}
	&& \alpha\big(x_1^{m_1}x_2^{m_2}\ldots x_t^{m_t}y_{j_1}^{k_{j_1}}y_{j_2}^{k_{j_2}}\ldots y_{j_t}^{k_{j_t}},
	x_1^{m'_1}x_2^{m'_2}\ldots x_t^{m'_t}y_{j_1}^{k'_{j_1}}y_{j_2}^{k'_{j_2}}\ldots y_{j_t}^{k'_{j_t}}\big)\\
	=&& \prod_{1 \leq i \leq t} \xi^{c(i,j_i)m_i'k_{j_i}} \text{ for  } 0 \leq c(i,j_i) \leq p-1.
\end{eqnarray*}
Our aim is to construct an irreducible $\alpha$-representation of $G$.
In this case all the irreducible $\alpha$-representations are $p^t$ dimensional.

For $1\leq i \leq t$, 
consider the subgroups $H_i=\mathbb Z/p^{r_i}\mathbb Z \times \mathbb Z/p\mathbb Z =\langle x_i \rangle  \times  \langle y_{j_i} \rangle$ of $G$.  	Every cocycle of $H_i$ are cohomologous to the cocycles of the form
\[
\beta(x_i^{l}y_{j_i}^{q},x_i^{l'}y_{j_i}^{q'})=\alpha|_{H_i \times H_i}(x_i^{l}y_{j_i}^{q},x_i^{l'}y_{j_i}^{q'})=\xi^{c(i,j_i){l'q}}, 0\leq  c(i,j_i) \leq (p-1),
\]
follows from \cite[Lemma 2.2$(i)$]{PS}.
In the next result  we describe an irreducible $\beta$-representation of $H_i$. 
\begin{thm}\label{alternate}
	For a non-trivial  cocycle $\beta$ of $H_i, 1 \leq i \leq t$, there is an irreducible $\beta$-representation $\Gamma_i$ of $H_i$ of dimension $p$  which is defined on the generators as $p \times p$ matrix by the following: for $1\leq m, n \leq p$,
	\begin{equation*}
	\Gamma_i(x_i)_{mn} = \begin{cases}
	\xi^{-(m-1)c(i,j_i)} &\text{if ~ $m=n,$}\\
	0 &\text{otherwise.}
	\end{cases}
	\end{equation*}
	\begin{equation*}
	\Gamma_i(y_{j_i})_{mn}  = \begin{cases}
	1 &\text{if  $m-n=1 \pmod p,$}\\
	0 &\text{otherwise.}
	\end{cases}
	\end{equation*}
\end{thm}
\begin{proof}
The cocycle $\beta$ represents the matrix $A_\beta=\left(
 \begin{array}{c c}
 0 & c(i,j_i)\\
 -c(i,j_i) & 0
 \end{array}
 \right)$. Now $\beta |_{\langle x_i \rangle \times \langle x_i \rangle}=1.$ Consider the trivial representation $\rho$ of $\langle x_i \rangle$. $T=\{y_{j_i}^{k}, 0 \leq k \leq p-1\}$ is a transversal of $ \langle x_i \rangle$ in $H_i$, and $\rho^{(y_{j_i}^{k})}(x_i)=\beta^{-1}(y_{j_i}^k, x_i)=\xi^{-c(i,j_i)k} \neq 1$ for $k \neq 0$. Hence, by Theorem \ref{Mackey}, 
 $\Gamma_i= \ind^{H_i}_{\langle x_i \rangle }(\rho)$ is an irreducible $\beta$-representation of $H_i$. Hence result follows.

%
\end{proof}
To construct an irreducible $\alpha$-representation of $G$, consider the irreducible $\beta=\alpha|_{H_i \times H_i}$-representations $\Gamma_i$ of $H_i$. Then the irreducible $\alpha$-representation $\gamma$ of $G$ is defined on the generators as $p^t \times p^t$ matrix by
\[
\gamma(x_i) =\Gamma_1(1) \otimes  \cdots \otimes  \Gamma_i(x_i)  \otimes \cdots   \otimes\Gamma_t(1),~1\leq i \leq t,
\]
\[
\gamma(y_{j_i}) =\Gamma_1(1) \otimes \cdots \Gamma_i(y_{j_i})\otimes \cdots   \otimes\Gamma_t(1).
\]

\bigskip

\begin{example}
Let $G=\mathbb Z/4\mathbb Z \times \mathbb Z/4\mathbb Z \times \mathbb Z/2\mathbb Z \times \mathbb Z/2\mathbb Z =\langle x_1\rangle \times \langle x_2\rangle \times \langle y_1\rangle \times  \langle y_2 \rangle $. Consider a bilinear cocycle of the form
\[
\alpha(x_1^{m_1}x_2^{m_2}y_1^{k_1}y_2^{k_2}, x_1^{m_1'}x_2^{m_2'}y_1^{k_1'}y_2^{k_2'} )=(-1)^{m_1'k_2+m_2'k_1}.
\]
Then it is easy to check that
\[
G_0=\langle x_1^2, x_2^2 \rangle.
\]
Using the construction of section \ref{s=t},  there is an irreducible $\alpha$-representation of $G$ of dimension $4$ which is defined on the generators as a $4\times 4$ matrix  as follows.\\

 $\gamma(x_1)= diag(1,-1)\otimes I_2$,  $\gamma(x_2)=  I_2 \otimes diag(1,-1)$, \\
 
 $\gamma(y_1)= I_2 \otimes \begin{bmatrix}
    0 & 1  \\
    1 &  0
  \end{bmatrix}$, $\gamma(y_2)=
  \begin{bmatrix}
    0 & 1  \\
    1 &  0
  \end{bmatrix} \otimes I_2$.

\end{example}

\bigskip

\subsection{General set up}\label{general}
If $G=G_1 \times G_2$ such that  $G_1, G_2$ are any abelian $p$-groups of rank $t$ and  $s$ respectively. Write 
\[
G_1=\langle x_1 \rangle \times \langle x_2 \rangle \times \ldots  \times \langle x_t \rangle \cong \mathbb Z/p^{r_1}\mathbb Z \times \mathbb Z/p^{r_2}\mathbb Z \times \cdots \times  \mathbb Z/p^{r_t}\mathbb Z 
\]
such that  $r_1 \geq r_2 \ldots \geq r_t \geq 1$ and
\[
G_2=\langle y_1 \rangle \times \langle y_2 \rangle \times \ldots  \times \langle y_s \rangle \cong \mathbb Z/p^{m_1}\mathbb Z \times \mathbb Z/p^{m_2}\mathbb Z \times \cdots \times  \mathbb Z/p^{m_s}\mathbb Z 
\]
such that  $m_1 \geq m_2 \ldots \geq m_s \geq 1$. There is a normal subgroup $N$ of $G_2$ such that $G_1\times (G_2/N) \cong G_1\times (\mathbb Z/p\mathbb Z )^s$.
Consider the corresponding inflation map $$\inf: \Ho^2(G_1\times G_2/N , \mathbb C^\times) \to \Ho^2(G_1\times G_2, \mathbb C^\times).$$
Suppose $\alpha$ is a   bilinear cocycle  of $G$ such that $[\alpha]$ is of order $p$.  
Then using  \eqref{bilinearcocycle} we can write $\alpha$ explicitly,  
and  it follows that $[\alpha]=\inf ([\beta])$ such that $\beta$ is a bilinear cocycle of $G_1\times (\mathbb Z/p\mathbb Z )^s$ of order $p$ of the form \eqref{bilinea}.
Then one can construct an irreducible $\alpha$-representation of $G$ by constructing an irreducible $\beta$-representation of $G_1\times (\mathbb Z/p\mathbb Z )^s$, follows from \cite[Theorem 3.2]{PS}. The construction of an $\beta$-representation of $G_1\times (\mathbb Z/p\mathbb Z )^s$ is given in section \ref{construction}.


 \bigskip

\begin{example}
Let $G=\mathbb Z/p^2\mathbb Z \times \mathbb Z/p^2\mathbb Z \times \mathbb Z/p^2\mathbb Z \times \mathbb Z/p^2\mathbb Z=\langle x_1\rangle \times \langle x_2\rangle \times \langle y_1\rangle \times  \langle y_2 \rangle $. 

(1) Consider a bilinear cocycle of the form
\[
\alpha(x_1^{m_1}x_2^{m_2}y_1^{k_1}y_2^{k_2}, x_1^{m_1'}x_2^{m_2'}y_1^{k_1'}y_2^{k_2'} )=\xi^{m_1'(k_1+k_2)}, \xi=e^{\frac{2\pi i}{p}}.
\]
Now we define the map $\tau$ on the generators of $G$ as defined in Example \ref{example1} (1). Then $\tau \in \irr^\alpha(G)$.

 \bigskip

(2) Now if we consider another bilinear cocycle of $G$  of the form
\[
\alpha(x_1^{m_1}x_2^{m_2}y_1^{k_1}y_2^{k_2}, x_1^{m_1'}x_2^{m_2'}y_1^{k_1'}y_2^{k_2'} )=\xi^{k_1(m_1'+m_2')}, \xi=e^{\frac{2\pi i}{p}}.
\]
Define the map $\tau$ on the generators of $G$ as defined in Example \ref{example1} (2). Then $\tau \in \irr^\alpha(G)$.

\end{example}

\section{Restriction of the $\alpha$-representations on subgroups of index  $\leq p^2$}\label{restriction}
Let $Z(p,m,n,r)$ denotes the number of matrices in $M_{m\times n}(\mathbb Z/p\mathbb Z)$ of rank $r$. Then  
$$
Z(p,m,n,r)=\Big(\prod_{j=1}^{r}(p^n-p^{j-1}) \Big)\Big(\sum_{\substack{j_1, j_2, \ldots j_r \in \mathbb{Z}_{\geq 0}, \\ j_1+j_2+\cdots +j_r \leq m-r}}p^{\sum_{i=1}^r i j_i} \Big).
$$
Suppose $G=G_1\times G_2$ is given in  \eqref{elementaryabelian}, and $\alpha$ is a bilinear cocycle of $G$ of the form \eqref{bilinea}. Let $H$ be a subgroup of $G$ of index $\leq p^2$. In this section,  we count the number of cohomology classes of $[\alpha]$ of $G$ such that the restriction $\rho |_H$ behaves in the same way on $H$.

Observe that, if $H$ is a subgroup of $G_2$, then $\alpha|_{H\times H}$ is trivial, and hence the elements of $\mathrm{Irr}^{\alpha |_{H \times H}}(H)$ are one dimensional. 
Thus, this case is easy to handle. 
Hence, in the following results (Theorem \ref{indexp} and Theorem \ref{indexp2}), we consider $H$ to be a subgroup $G$ such that $H \not\subset G_2$.
\\

If $H$ is a subgroup of index $p$ of the group $G$  such that $H \not\subset G_2$, then, up to isomorphism, $H$ has one of the following forms.

$(i)$~There is a  $i$ such that $H=\langle x_i^{p} \rangle \times \prod_{k=1, k\neq i}^t \langle x_k \rangle \times G_2$. \hspace{5cm} (X)\\

$(ii)$~There is a  $j$ such that  $H=G_1 \times \prod_{r=1, r\neq j}^s \langle y_r \rangle $. 
 \hspace{6cm} (Y)\\

\begin{thm}\label{indexp}
Suppose $G=G_1\times G_2$ is of the form \eqref{elementaryabelian}, and $H$ is a subgroup of $G$ of index $p$. Let $\alpha$ be a bilinear cocycle of $G$ and $\rho\in \irr^\alpha(G)$. Then  the following holds.
	\begin{enumerate}
%
%
%
%

\item  Suppose $H$ is of the form $(X)$. Then
		
\begin{enumerate}[label=(\roman*)]
		\item there are $p^{n}Z(p,t-1,s,n)$ many cohomology classes $[\alpha]$  of $G$ such that $\rho |_H$ is irreducible and has degree $p^n$ for $n=0,1,2,\ldots, \min\{t-1,s\}$. 
		
		\item there are $p^{n}(p^{s-n}-1)Z(p,t-1,s,n)$ many cohomology classes  $[\alpha]$  of $G$ such that $\rho$ is $p^{n+1}$ dimensional and $$\rho |_H\cong   \oplus_{i=1}^{p} \rho_i,$$
	for $p$  distinct elements $\rho_i\in \irr^{\alpha |_{H\times H}}(H)$   for $n=0,1,2,\ldots, \min\{s,t\}-1$. 
	\end{enumerate}

\item  Suppose $H$ is of the form $(Y)$. Then
		
\begin{enumerate}[label=(\roman*)]
		\item there are $p^{n}Z(p,t,s-1,n)$ many cohomology classes $[\alpha]$  of $G$ such that $\rho |_H$ is irreducible and has degree $p^n$ for $n=0,1,2,\ldots, \min\{t,s-1\}$. 
		
		\item there are $p^{n}(p^{t-n}-1)Z(p,t,s-1,n)$ many cohomology classes  $[\alpha]$  of $G$ such that $$\rho |_H\cong   \oplus_{i=1}^{p} \rho_i,$$
	for $p$ distinct elements $\rho_i\in \irr^{\alpha |_{H\times H}}(H)$  for $n=0,1,2,\ldots, \min\{s,t\}-1$. 
	\end{enumerate}

	\end{enumerate}
%
%
%
%
\end{thm}
\begin{proof}

$(1)$ 
Suppose $H$ is of the form $(X)$.
 Let $\beta=\alpha|_{H\times H}$,
and $A_{1\beta}\in M_{(t-1)\times s}(\mathbb Z/p\mathbb Z)$  be of rank $n$ for  $n \leq \min\{s,t-1\}$. 
Let $A_{\beta}$ be a matrix of the form 
$
A_\beta
=
\left(
\begin{array}{c|c}
0 & A_{1\beta}\\
\hline
-A_{1\beta}^t & 0
\end{array}
\right).
$
Clearly  the result is true for $s+t \leq  2$.
	We assume that $s+t \geq 3$.
	There exist invertible matrices $P\in M_{(t-1)\times (t-1)}(\mathbb Z/p\mathbb Z)$ and $Q \in M_{s\times s}(\mathbb Z/p\mathbb Z)$ such that
	\[
PA_{1\beta} Q=
	\begin{pmatrix}
	 I_{n \times n} & 0 \\
	  0  &  0
	\end{pmatrix}.
	\]
	Let 
	$
	A_\alpha
	=
	\left(
	\begin{array}{c|c}
	0 & A_{1\alpha}\\
	\hline
	-A_{1\alpha}^t & 0
	\end{array}
	\right)
\in M_{(s+t)\times (s+t)}(\mathbb Z/p\mathbb Z)
$, where
	$A_{1\alpha}=
	\left(
	\begin{array}{c}
	 A_{1\beta} \\
	  D
	\end{array}
	\right)$ for  $D= 	\left(
	\begin{array}{c|c}
D_1 & 
	 D_2
	\end{array}
	\right)$ such that $D_1\in M_{1\times n}(\mathbb Z/p\mathbb Z), D_2\in M_{1 \times (s-n)}(\mathbb Z/p\mathbb Z)$.  
Then, 
we have
	\[
\left(
	\begin{array}{c|c}
	P & 0\\
	\hline
	0 & 1
	\end{array}
	\right)^T A_{1\alpha}Q
	=	
	\begin{pmatrix}
		 I_{n \times n} & 0  \\
	  0  &  0  \\
	  X  &  Y
	\end{pmatrix},
	\]
	where $X \in M_{1\times n}(\mathbb Z/p\mathbb Z)$, $Y \in M_{1\times (s-n)}(\mathbb Z/p\mathbb Z)$.
	After performing some row  operations on this matrix, we obtain
\[
\begin{pmatrix}
 I_{n \times n} & 0  \\
  0  &  0  \\
  0  &  Y 
\end{pmatrix}.
\]
Now the following two cases occur here.

$(i)$ Rank $(A_{1\alpha})=n$. In this case, $Y$ has to be the zero matrix. 
So, for a fix $A_{1\beta}$, the number of choices of $X, Y$  is $p^n$. \\

$(ii)$ Rank $(A_{1\alpha})=(n+1)$, provided $n \leq \min\{s,t-1\}, n+1 \leq \min\{s,t\}$. In this case, $Y$ is not the zero matrix. However, for a fix $X$, the number of choices of row $Y$ is  $p^{s-n}-1$.  
So, for a fix $A_{1\beta}$, the number of choices of $X, Y$  is $p^n(p^{s-n}-1)$. \\

Now result follows from \cite[Theorem 1.1, p. 268]{GK}.\\

$(2)$ 
If $H$ is of the form $(Y)$, then proof goes on the same lines as above by considering 
$A_{1\beta}\in M_{t\times (s-1)}(\mathbb Z/p\mathbb Z)$ and
$A_{1\alpha}=
	\left(
	\begin{array}{c| c}
	 A_{1\beta} &
	  D
	\end{array}
	\right)$ for  $D= 	\left(
	\begin{array}{c}
D_1\\ D_2
	\end{array}
	\right)$ such that $D_1\in M_{n\times 1}(\mathbb Z/p\mathbb Z), D_2\in M_{(t-n) \times 1}(\mathbb Z/p\mathbb Z)$.

%
%
%
%
\end{proof}

If $H$ is a subgroup of index $p^2$ of the group $G$ given in  \eqref{elementaryabelian} such that $H \not\subset G_2$, then upto isomorphism, $H$ has one of the following forms.

$(i)$~There are  $i, j$ with $i \neq j$ such that $H=\langle x_i^p \rangle \times \langle x_j^p \rangle \times \prod_{k=1, k\neq i,j}^t \langle x_k \rangle \times G_2$ provided $t>1$.  \hspace{13.6cm}  (A)\\

$(ii)$~There is a  $i$ such that $H=\langle x_i^{p^2} \rangle \times \prod_{k=1, k\neq i}^t \langle x_k \rangle \times G_2$. 
 \hspace{5cm} (B)\\

$(iii)$~There are  $i, j$ such that $H=  \langle x_i^p \rangle \times \prod_{k=1, k\neq i}^t \langle x_k \rangle \times \prod_{r=1, r\neq j}^s \langle y_r \rangle $.  \hspace{3cm} (C)\\

\begin{thm}\label{indexp2}
	Suppose $G=G_1\times G_2$, of the form \eqref{elementaryabelian} and $H$ is a subgroup of $G$ of index $p^2$. Let $\alpha$ be a non-trivial bilinear cocycle of $G$ and $\rho\in \irr^\alpha(G)$. Then  the following holds.
	\begin{enumerate}
%

	\item If $H$ is of the form $(A)$, then 
		\begin{enumerate}[label=(\roman*)]
	\item there are $p^{2n}Z(p,t-2,s,n)$ many cohomology classes $[\alpha]$  of $G$ such that $\rho |_H$ is irreducible and has degree $p^n$ for $n=0,1,2,\ldots,\min\{t-2,s\}.$

	\item   there are $p^{2n}(p+1)(p^{s-n}-1)Z(p,t-2,s,n)$ many cohomology classes $[\alpha]$  of $G$ such that $\rho$ is $p^{n+1}$ dimensional and 
	$$\rho |_H\cong   \oplus_{i=1}^{p} \rho_i,$$
	for $p$ distinct elements $ \rho_i \in\irr^{\alpha |_{H\times H}}(H)$ for $n=0,1,2,\ldots, \min\{t-1,s\}-1.$
		
	\item  there are $p^{2n}(p^{s-n}-1)(p^{s-n}-p) Z(p,t-2,s,n)$ many cohomology classes $[\alpha]$  of $G$ such that $\rho$ is $p^{n+2}$ dimensional and $$\rho |_H\cong   \oplus_{i=1}^{p^2} \rho_i,$$
	for $p^2$ distinct elements $ \rho_i \in\irr^{\alpha |_{H\times H}}(H)$ for $n=0,1,2,\ldots, \min\{t,s\}-2.$
	
%

			\end{enumerate}
			
			\item If $H$ is of the form $(B)$, then 		
			\begin{enumerate}[label=(\roman*)]
	\item there are $p^{n}Z(p,t-1,s,n)$ many cohomology classes $[\alpha]$  of $G$ such that $\rho |_H$ is irreducible and has degree $p^n$ for $n=0,1,2,\ldots,\min\{t-1,s\}.$

	\item   there are $p^{n}(p^{s-n}-1)Z(p,t-1,s,n)$ many cohomology classes $[\alpha]$  of $G$ such that $\rho$ is $p^{n+1}$ dimensional and $$\rho |_H\cong   \oplus_{i=1}^{p} \rho_i,$$
	for $p$ distinct elements $ \rho_i \in\irr^{\alpha |_{H\times H}}(H)$ for $n=0,1,2,\ldots, \min\{t,s\}-1.$
			\end{enumerate}
			
					\item If $H$ is of the form $(C)$, then 
		\begin{enumerate}[label=(\roman*)]
	\item there are $p^{2n}Z(p,t-1,s-1,n)$ many cohomology classes $[\alpha]$  of $G$ such that $\rho |_H$ is irreducible and has degree $p^n$ for $n=0,1,2,\ldots,\min\{t-1,s-1\}.$

	\item   there are $p^{2n+1}[(p^{t-n-1}-1)+(p^{s-n-1}-1)]Z(p,t-1,s-1,n)$ many cohomology classes $[\alpha]$  of $G$ such that $\rho$ is $p^{n+1}$ dimensional and $$\rho |_H\cong   \oplus_{i=1}^{p} \rho_i,$$
	for $p$ distinct elements $ \rho_i \in\irr^{\alpha |_{H\times H}}(H)$ for $n=0,1,2,\ldots,\min\{t,s\}-1.$
	
	\item  there are $p^{2n+1}(p^{t-n-1}-1)(p^{s-n-1}-1) Z(p,t-1,s-1,n)$ many cohomology classes $[\alpha]$  of $G$ such that $\rho$ is $p^{n+2}$ dimensional and $$\rho |_H\cong   \oplus_{i=1}^{p^2} \rho_i,$$
	for $p^2$ distinct elements $ \rho_i \in\irr^{\alpha |_{H\times H}}(H)$ for $n=0,1,2,\ldots, \min\{t,s\}-2.$

			\end{enumerate}
			\end{enumerate}
\end{thm}
\begin{proof}
$(1)$
 Assume that there are  $i, j$ such that $i \neq j$ and $H=\langle x_i^p \rangle \times \langle x_j^p \rangle \times \prod_{k=1, k\neq i,j}^t \langle x_k \rangle \times G_2$. In this case, $t>1$. Now the result is true for $s+t \leq 2$.
 We assume that $s+t \geq 3$. 
Let $\beta=\alpha|_{H \times H}$
 and $A_{1\beta} \in M_{(t-2)\times s}(\mathbb Z/p\mathbb Z)$ be a matrix of  rank $n$. Then there exist invertible matrices $P\in M_{(t-2)\times (t-2)}(\mathbb Z/p\mathbb Z)$ and $Q \in M_{s\times s}(\mathbb Z/p\mathbb Z)$ such that
	\[
PA_{1\beta} Q=
	\begin{pmatrix}
	 I_{n \times n} & 0 \\
	  0  &  0
	\end{pmatrix}.
	\]	
 	Let 	\[
 A_\alpha
 =
 \left(
 \begin{array}{c|c}
 0 & A_{1\alpha}\\
 \hline
 -A_{1\alpha}^t & 0
 \end{array}
 \right)
 \in M_{(s+t)\times (s+t)}(\mathbb Z/p\mathbb Z),
 \]  where
 $A_{1\alpha}=
 \left(
 \begin{array}{c}
 A_{1\beta} \\
 D
 \end{array}
 \right)$ for some $D= 	\left(
 \begin{array}{c c}
 D_1 &D_3 \\
 D_2 & D_4
 \end{array}
 \right)$ such that $D_1, D_2\in M_{1\times n}(\mathbb Z/p\mathbb Z)$ and $ D_3, D_4\in M_{1\times (s-n)}(\mathbb Z/p\mathbb Z)$.   
 Then
$(P \oplus I_{2\times 2})^T A_{1\alpha} Q$ is of the form
   \[
 \begin{pmatrix}
  I_{n \times n} & 0  \\
   0  &  0  \\
  X_1 & X_3\\
  X_2 & X_4
 \end{pmatrix},
 \]
 where $X_1, X_2 \in  M_{1\times n}(\mathbb Z/p\mathbb Z)$ and $X_3, X_4 \in M_{1\times (s-n)}(\mathbb Z/p\mathbb Z)$.
 After performing some row  operations on this matrix, we obtain 
 \[
\begin{pmatrix}
 I_{n \times n} & 0  \\
  0  &  0  \\
 0 & X_3\\
 0 & X_4
\end{pmatrix}
\]
 Now we discuss the possible three cases here.\\
 
\textbf{Case 1.} Rank $A_{1\alpha}=n$ if and only if $X_3, X_4$  are zero matrices. 
 Therefore, for a fix $A_{1\beta}$, the number of matrices $A_{1\alpha}$ of rank $n$ is the number of choices of the matrices $X_1, X_2$ which is $p^{2n}$. 
 \\
 
\textbf{Case 2.}  Rank $A_{1\alpha}=n+1$ provided 
$n\leq \min\{t-2,s\}$ and $n+1\leq  \min\{t,s\}$. In this case,  either $X_3 \neq 0$ and $X_4$ is scalar multiple of $X_3$, or $X_3=0$ and $X_4\neq 0$.
 If   $X_3 \neq 0$ and $X_4$ is a scalar multiple of $X_3$, then the number of choices of $X_3$, $X_4$ is $p(p^{s-n}-1)$.
If  $X_3=0$, $X_4\neq 0$, then the number of choices of $X_3, X_4$ is $(p^{s-n}-1)$.
    Hence, the number of matrices $A_{1\alpha}$ of rank $n+1$, for a fix $A_{1\beta}$, is $p^{2n}(p+1)(p^{s-n}-1)$.

Now set of all $\alpha$-regular elements of $G$ is not the same as the set of  all $\beta$-regular elements of $H$, as those are determined by the kernel of the matrices $A_{1\alpha}$ and $A_{1\beta}$ respectively.  Hence, the result follows by \cite[Proposition 3.3]{RH}.
    \\
 
\textbf{Case 3.} Rank $A_{1\alpha}=(n+2)$ provided 
$n\leq \min\{t-2,s\}$ and $n+2\leq  \min\{t,s\}$. In this case,  
$\begin{pmatrix}X_3 \\  X_4 \end{pmatrix} $
 must be  of rank $2$. Observe that $X_3$ or $X_4$  can not be zero matrices, and $X_4$ is not a scalar multiple of $X_3$. So the number of choices of $X_3$, $X_4$ is $(p^{s-n}-1)(p^{s-n}-p)$. Hence, for a fix $A_{1\beta}$, the number of matrices $A_{1\alpha}$ of rank $(n+2)$ is  $p^{2n}(p^{s-n}-1)(p^{s-n}-p)$.\\
%

$(2)$ If $H$ is of the form $(B)$, then take $A_{1\beta} \in M_{(t-1)\times s}(\mathbb Z/p\mathbb Z)$ of rank $n$ and $A_{1\alpha}=
 \left(
 \begin{array}{c}
 A_{1\beta} \\
 D
 \end{array}
 \right)$ for some $D= 	\left(
\begin{array}{c|c}
D_1 & 
	 D_2
	\end{array}
 \right)$ such that $D_1\in M_{1\times n}(\mathbb Z/p\mathbb Z), D_2\in M_{1\times (s-n)}(\mathbb Z/p\mathbb Z)$.   
 Then
$(P \oplus 1)^T A_{1\alpha} Q$ is of the form
   \[
 \begin{pmatrix}
  I_{n \times n} & 0  \\
   0  &  0  \\
  X & Y
 \end{pmatrix},
 \]
 where $X \in  M_{1\times n}(\mathbb Z/p\mathbb Z)$ and $Y \in M_{1\times (s-n)}(\mathbb Z/p\mathbb Z)$.
 After performing some row  operations on this matrix, we obtain 
 \[
\begin{pmatrix}
 I_{n \times n} & 0  \\
  0  &  0  \\
 0 & Y
\end{pmatrix}.
\]
Now proof goes on the same lines as in the previous case.
\\

(3) If $H$ is of the form $(C)$, then take $A_{1\beta} \in M_{(t-1)\times (s-1)}(\mathbb Z/p\mathbb Z)$ of rank $n$ and $A_{1\alpha}=
 \left(
 \begin{array}{c c}
 A_{1\beta} & D_1\\
 D_2 &D_3
 \end{array}
 \right)$ for  $D_1\in M_{(t-1)\times 1}(\mathbb Z/p\mathbb Z), D_2\in M_{1\times (s-1)}(\mathbb Z/p\mathbb Z), D_3 \in  \mathbb Z/p\mathbb Z$.   Then there are invertible matrices $P, Q$ such that $PA_{1\alpha}Q$ is of the form   \[
 \begin{pmatrix}
  I_{n \times n} & 0  & X_1 \\
   0  &  0  & X\\
  Y_1 & Y & Z
 \end{pmatrix},
 \]
   where $X_1 \in  M_{n\times 1}(\mathbb Z/p\mathbb Z), Y_1 \in M_{1\times n}(\mathbb Z/p\mathbb Z)$, $X \in  M_{(t-n-1)\times 1}(\mathbb Z/p\mathbb Z), Y \in M_{1\times (s-n-1)}(\mathbb Z/p\mathbb Z)$, and $Z\in \mathbb Z/p\mathbb Z$.
Then after performing row and column operations, 
finally we have 
   \[
 \begin{pmatrix}
  I_{n \times n} & 0  & 0 \\
   0  &  0  & X\\
  0 & Y & Z_1
 \end{pmatrix}.
 \]
Now proof goes on the same lines as in the previous cases.
\end{proof}

\section{Construction for Direct product of any two $p$-groups}\label{directproduct}
Let 
\begin{eqnarray}\label{anydirect}
G=G_1\times G_2
\end{eqnarray} such that 
\[
G_1/G_1'=\langle x_1 \rangle \times \langle x_2 \rangle \times \ldots  \times \langle x_t \rangle \cong \mathbb Z/p^{r_1}\mathbb Z \times \mathbb Z/p^{r_2}\mathbb Z \times \cdots \times  \mathbb Z/p^{r_t}\mathbb Z,~r_1 \geq r_2 \ldots \geq r_t \geq 1,
\]
\[
G_2/G_2'=\langle y_1 \rangle \times \langle y_2 \rangle \times \ldots  \times \langle y_s \rangle \cong \mathbb Z/p^{n_1}\mathbb Z \times \mathbb Z/p^{n_2}\mathbb Z \times \cdots \times  \mathbb Z/p^{n_s}\mathbb Z, 
~  n_1 \geq n_2 \ldots \geq n_s \geq 1.
\]

\bigskip

\begin{thm}\label{anydirectproduct}
	Let $G=G_1\times G_2$ be of the form \eqref{anydirect}, and $\alpha$ be a bilinear cocycle of $G$. Then there is a  bilinear cocycle $\beta$ of  $G_1/G_1' \times G_2/G_2'$  such that  $\inf([\beta])=[\alpha]$, which determines an isomorphism between the cohomology classes of bilinear cocycles of $G_1/G_1' \times G_2/G_2'$ and $G_1\times G_2$.	
\end{thm}
\begin{proof}
Since by \eqref{Schur}, 
$$\Ho^2(G_1/G_1' \times G_2/G_2', \mathbb C^\times)\cong \Ho^2(G_1/G_1' , \mathbb C^\times)\times  \Ho^2(G_1/G_2' , \mathbb C^\times)\times \Hom( G_1/G_1' \otimes G_2/G_2', \mathbb C^\times)$$ and 
$$\Ho^2(G_1\times G_2, \mathbb C^\times)\cong \Ho^2(G_1, \mathbb C^\times)\times  \Ho^2(G_2, \mathbb C^\times)\times \Hom( G_1/G_1' \otimes G_2/G_2', \mathbb C^\times),$$
 the result follows.
\end{proof}
If we consider any element $\eta \in \irr^\beta(G_1/G_1' \times G_2/G_2')$, then we can define $\rho \in \irr^\alpha(G_1 \times G_2)$ by 
$\rho:=\eta \circ \pi$ for the projection map $\pi: G_1 \times G_2 \to G_1/G_1' \times G_2/G_2'$, follows from Theorem \ref{anydirectproduct}.

\begin{enumerate}
\item It follows from Theorem  \ref{anydirectproduct} that the construction of an element of $\irr^\alpha(G)$  such that $[\alpha]$ is of order $p$, 
boils down to the construction of an element of $\irr^\beta(H)$ for the abelian group $H=G_1/G_1' \times G_2/G_2'$  such that $[\beta]$ is of order $p$. Then the construction follows from section \ref{general}.\\

\item Now if  $G=G_1\times G_2$ be of the form  \eqref{anydirect} such that  $G_2/G_2' \cong (\mathbb Z/p\mathbb Z)^s$. It is easy to see  that  $[\alpha]$  are of order $p$ for the bilinear cocycle $\alpha$ of $G$, and $[\beta]$ are of the form given in  \eqref{bilinea}. By Theorem  \ref{anydirectproduct}, it is enough to construct an element of $\irr^\beta(H)$ for the abelian group $H=G_1/G_1' \times (\mathbb Z/p\mathbb Z)^s$. Then the construction follows from section \ref{construction}.
\end{enumerate}

\begin{example}
Suppose $G=G_1\times G_2$, where $G_1$ and $G_2$ are of nilpotency class $2$ having the following presentations
\begin{eqnarray*}
G_1&=&\langle x_1,x_2 \mid  [x_1,x_2]=z, x_i^{p^r}=z^{p^r}=1, i=1,2\rangle,\\
G_2&=&\langle  y_1,y_2 \mid  [y_1,y_2]=y_1^p, y_1^{p^2}=y_2^{p}=1\rangle. 
\end{eqnarray*}

(1) Consider a bilinear cocycle of $G$ of the form
\[
\alpha(x_1^{m_1}x_2^{m_2}y_1^{k_1}y_2^{k_2}z^m, x_1^{m_1'}x_2^{m_2'}y_1^{k_1'}y_2^{k_2'}z^n )=\xi^{m_1'(k_1+k_2)}, \xi=e^{\frac{2\pi i}{p}}.
\]
Now we define the map $\tau$ on the generators of $G$ as defined in Example \ref{example1} (1). From the above discussion, it follows that $\tau \in \irr^\alpha(G)$.

 \bigskip

(2) Consider another bilinear cocycle of $G$  of the form
\[
\alpha(x_1^{m_1}x_2^{m_2}y_1^{k_1}y_2^{k_2}z^m, x_1^{m_1'}x_2^{m_2'}y_1^{k_1'}y_2^{k_2'} z^n)=\xi^{k_1(m_1'+m_2')}, \xi=e^{\frac{2\pi i}{p}}.
\]
Define the map $\tau$ on the generators of $G$ as defined in Example \ref{example1} (2). From the above discussion, it follows that $\tau \in \irr^\alpha(G)$.

\end{example}

%
%
%
%


\bibliographystyle{amsplain}
\bibliography{projective}

\end{document}